\documentclass{amsart}

\usepackage{amsthm, amssymb, amsmath}
\usepackage{a4wide}
\usepackage[dvipsnames]{xcolor}
\usepackage[english]{babel}
\usepackage{graphicx, tikz, pgfplots}
\usepackage{subfigure}
\usepackage{url}
\usepackage{dsfont}
 \usepackage[foot]{amsaddr}
\usepackage{fancyhdr}
\newcommand{\R}{\mathbb{R}}

\newcommand{\N}{\mathbb{N}}
\newcommand{\I}{\mathbb{I}}

\newcommand{\B}{\mathcal{B}}
\newcommand{\F}{\mathcal{F}}

\newcommand{\G}{\mathcal{G}}
\newcommand{\Eb}{\mathbb{E}}

\newcommand{\lbetar}{\lfloor \beta \rfloor}

\DeclareMathOperator{\var}{var}
\DeclareMathOperator{\bow}{Bow}
\DeclareMathOperator*{\interior}{int}
\newtheorem{thm}{Theorem}[section]
\newtheorem{cor}[thm]{Corollary}
\newtheorem{lem}[thm]{Lemma}
\newtheorem{prop}[thm]{Proposition}
\newtheorem{defn}{Definition}[section]
\newtheorem{rem}[equation]{Remark}
\begin{document}

\title{Equilibrium States for the Random $\beta$-Transformation through $g$-Measures}
\author{Karma Dajani and Kieran Power}
\address[Karma Dajani]{Department of Mathematics, Utrecht University, P.O.~Box 80010, 3508TA Utrecht, the Netherlands,}
\email[Karma Dajani]{k.dajani1@uu.nl}
\address[Kieran Power]{Department of Mathematics, Utrecht University, P.O.~Box 80010, 3508TA Utrecht, the Netherlands}
\address[Kieran Power]
{School of Mathematics, Statistics and Actuarial Science, University of Kent, Canterbury Kent CT2 7FS, United Kingdom}
\email[Kieran Power]{kcp6@kent.ac.uk}




\maketitle

\begin{abstract}
We consider the random $\beta$-transformation $K_{\beta}$, defined on $\{0,1\}^{\N}\times[0, \frac{\lfloor\beta\rfloor}{\beta-1}]$, that generates all possible expansions of the form $x=\sum_{i=0}^{\infty}\frac{a_i}{\beta^i}$, where 
$a_i\in \{0,1,\cdots,\lfloor\beta\rfloor\}$. This transformation was introduced in \cite{KarmadeVries, KarmaInv, KarmaKraaikampRandom}, where two natural invariant ergodic measures were found. The first is the unique measure of maximal entropy, and the second is a measure of the form $m_p\times \mu_{\beta}$, with $m_p$ the Bernoulli $(p,1-p)$ product measure and $\mu_{\beta}$ is a measure equivalent to Lebesgue measure. In this paper, we give an uncountable family of $K_{\beta}$-invariant exact $g$-measures for a certain collection of algebraic $\beta$'s.
\end{abstract}

\section{Introduction}\label{intro}
Let $\beta >1$ be a non-integer and let $I_{\beta} = \big[0, \frac{\lbetar}{\beta -1} \big]$. It is a well known fact that Lebesgue almost all $x \in I_\beta$ have uncountably many different $\beta$-expansions (see \cite{Sidorov}). These expansions can all be generated by iterating a certain dynamical system, called the random $\beta$-transformation. Define the maps $\{ T_k \}_{0 \le k \le \lbetar}$ by $T_k (x) = \beta x -k$. These maps together partition the interval $I_{\beta}$ in a natural way: 
\begin{align*}
E_0 &= \Big[ 0, \frac{1}{\beta} \Big), \quad E_{\lbetar} = \Big( \frac{\lbetar}{\beta(\beta-1)} + \frac{\lbetar -1}{\beta}, \frac{\lbetar}{\beta-1} \Big],\\
E_k &= \Big( \frac{\lbetar}{\beta(\beta-1)} + \frac{k-1}{\beta}, \frac{k+1}{\beta} \Big), \quad 1 \le k \le \lbetar-1,\\ 
S_k &= \Big[ \frac{k}{\beta}, \frac{\lbetar}{\beta(\beta-1)} + \frac{k-1}{\beta} \Big], \quad 1 \le k \le \lbetar.
 \end{align*}
On $\{0,1\}^{\mathbb N} \times I_{\beta}$ we define the {\em random $\beta$-transformation} $K_{\beta}$ by
\[ K_{\beta} (\omega, x) = \left\{
\begin{array}{ll}
\big(\omega, T_k (x)\big), & \text{if } x \in E_k, \, 0 \le k \le \lbetar,\\
\\
\big(\sigma (\omega), T_{k-1+\omega_1}(x) \big), & \text{if } x \in S_k, \, 1 \le k \le \lbetar,
\end{array}\right.\]
where $\sigma$ denotes the left shift on sequences, i.e., $\sigma (\omega_n)_{n \ge 1} = (\omega_{n+1})_{n \ge 1}$. 
\begin{figure}[ht]
\begin{tikzpicture}[scale=4.5]
\filldraw[fill=yellow!20, draw=yellow!20](.285,0) rectangle (.339,1.1927);
\filldraw[fill=yellow!20, draw=yellow!20](.569,0) rectangle (.624,1.1927);
\filldraw[fill=yellow!20, draw=yellow!20](.853,0) rectangle (.908,1.1927);
\draw(0,0)--(.285,0)node[below]{$\frac1{\beta}$}--(.569,0)node[below]{$\frac2{\beta}$}--(.853,0)node[below]{$\frac3{\beta}$}--(1.1927,0)node[below]{$\frac3{\beta-1}$}--(1.1927,1.1927)--(1.05,1.1927)node[above]{$E_3$}--(.875,1.1927)node[above]{$S_3$}--(.73,1.1927)node[above]{$E_2$}--(.595,1.1927)node[above]{$S_2$}--(.46,1.1927)node[above]{$E_1$}--(.308,1.1927)node[above]{$S_1$}--(.14,1.1927)node[above]{$E_0$}--(0,1.1927)--(0,0);
\draw[thick, purple!50!black](0,0)node[below]{0}--(.339,1.1927)(.285,0)--(.624,1.1927)(.569,0)--(.908,1.1927)(.853,0)--(1.1927,1.1927);
\draw[dotted](0,0)--(1.1927,1.1927)(.285,0)--(.285,1.1927)(.339,0)--(.339,1.1927)(.569,0)--(.569,1.1927)(.624,0)--(.624,1.1927)(.853,0)--(.853,1.1927)(.908,0)--(.908,1.1927);
\draw[red] (.285,1)--(1,1)--(1,.515)--(.515,.515)--(.515,.812)--(.812,.812)--(.812,.853)--(.853,.853);
\draw[green] (.908,.1927)--(.1927,.1927)--(.1927,.6777)--(.6777,.6777)--(.6777,.3807)--(.3807,.3807)--(.3808,.339)--(.339,.339);
\end{tikzpicture}
\caption{The intervals $E_k$ and $S_k$ for $\beta \approx 3.515$ given by $\beta^4-3\beta^3-\beta^2-2\beta-3=0$. The red and green lines indicate the orbits of the points 1 and $\frac1{\beta-1}-1$ respectively.}
\label{f:randommap}
\end{figure}
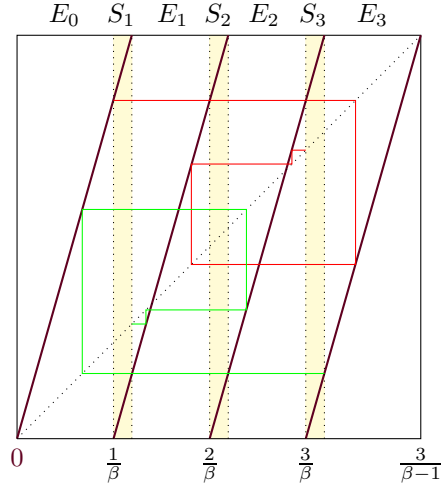
See Figure~\ref{f:randommap} for a picture.  To see the expansion, we first define
\[d_1=d_1(\omega,x):=\begin{cases}
k, & \text{if}~x\in E_k~\text{for}~k\in{0,\dots,\lfloor{\beta}\rfloor}\\
& \text{or} ~(\omega,x)\in \{\omega_1=1\}\times S_k~\text{for}~k\in{1,\dots,\lfloor{\beta}\rfloor}.\\
k-1,&\text{if}~ (\omega,x)\in \{\omega_1=0\}\times S_k~\text{for}~k\in{1,\dots,\lfloor{\beta}\rfloor}.
\end{cases}\]
Define, for each $n\in \N$,
$d_n=d_n(\omega,x):=d_1(K_\beta^{n-1}(\omega,x))$, and let $\pi_2:\Omega\times I_{\beta}\to I_{\beta}$ be the canonical projection onto the second coordinate. Thus
\begin{equation}\label{projeq}
\pi_2(K_\beta^n(\omega,x))=	\beta^nx-\sum_{i=1}^n\beta^{n-i}d_i,
\end{equation}
giving that:
\[x=\sum_{i=1}^n\frac{d_i}{\beta^i}+\frac{\pi_2(K_\beta^n(\omega,x))}{\beta^n}.\]
As $0\leq\pi_2(K_\beta^n(\omega,x))\leq\lfloor{\beta}\rfloor/(\beta-1)$ for all $n$, we thus have
\begin{equation}\label{xrepKbeta}
x=\sum_{n=1}^\infty\frac{d_n}{\beta^n}=\sum_{n=1}^\infty\frac{d_n(\omega,x)}{\beta^n}.
\end{equation}
Define $\varphi: \Omega\times I_{\beta} \to  \{ 0,1, \ldots, \lbetar \}^{\mathbb N}$ by
\[\varphi(\omega,x)=(d_1(\omega,x),d_2(\omega,x),\cdots).\]
In \cite{KarmadeVries} it was shown that if $m$ is the uniform Bernoulli measure on  $\{ 0,1, \ldots, \lbetar \}^{\mathbb N}$, then $\nu_\beta= m \circ \varphi$ is the unique measure of maximal entropy for $K_{\beta}$, with entropy $\log \lceil \beta \rceil$. Moreover, $\varphi$ is a measurable isomorphism between $K_{\beta}$, under the measure $\nu_{\beta}$, and $\sigma$ on $\{ 0,1, \ldots, \lbetar \}^{\mathbb N}$, under the measure $m$.
Furthermore, in \cite{KarmaInv}, it was shown that $K_{\beta}$ also admit an invariant ergodic measure of the form $m_p\times \mu_{\beta}$, where $\mu_{\beta}$ is equivalent to Lebesgue measure on $I_{\beta}$.
\vskip .2cm

Our aim is to give an uncountable family of $K_{\beta}$-invariant $g$-measures. This will be done for a special class of $\beta$'s that will be introduced in Section \ref{MarkovChain}. In Section \ref{gMTO}, we give a brief overview of $g$-measures and we state the important theorems that will be used in this article. In Section \ref{gmeasures}, an explicit family of $g$-measures is exhibited.

\section{$g$-Measures and the Transfer Operator}\label{gMTO}

The first appearance of $g$-measures as an object of strictly mathematical study occurred in Michael Keane's seminal paper \textit{Strongly mixing g-measures} \cite{Keane}. The objects of interest are dynamical systems $(X,T)$, where $(X,d)$ is a compact metric space, and $T$ is a $n\geq 2$ covering transformation.  That is $T$ satisfies the following four conditions:
\begin{itemize}
\item[(i)] $T$ is $n$-to-one.\\
\item[(ii)] $T$ is a local homeomorphism.\\
\item[(iii)] There exists a constant $C>1$ and $\delta>0$ such that for any $x,y\in X$ with $d(x,y)<\le \delta$, we have $d(Tx,Ty)\ge Cd(x,y)$.\\
\item[(iv)] For each $\epsilon>0$ there is $N_{\epsilon}\in \N$ such that $T^{-n}(x)$ is $\epsilon$ dense for every $x\in X$  and $n\ge N_{\epsilon}$..
\end{itemize}
In this paper $g$-measures are defined to be probability measures $\mu$ satisfying a certain Radon-Nikodym derivative relation with respect to a locally lifted (by $T$) measure $Q\mu$. In particular $\frac{d\mu}{dQ\mu}=g$, where $g$ is a member of the class
\[\G:=\{g:X\to[0,1]:\sum_{y\in T^{-1}x}g(y)=1~\mu~\text{a.e.}\}.\]
Members of $\G$ are known as $g$-functions, although in the literature this nomenclature is sometimes reserved exclusively for continuous members of $\mathcal G$. We shall denote continuous members of $\mathcal G$ by $\G_c$.\\
In contemporary literature $g$-measures are generally defined with respect to the transfer operator. For a $\phi\in C(X)$ the associated transfer operator $L_\phi:C(X)\to C(X)$ is defined by
\begin{equation}\label{deftransferop}
L_\phi f(x)=\sum_{y\in T^{-1}x}e^{\phi(y)}f(y).
\end{equation}
It is easy to verify that $L_\phi$ is a bounded linear operator for every $\phi\in C(X)$.
For any $\mu\in M(X)$ and $\phi\in C(X)$ there exists, by the Riesz-Representation, a unique measure $\mu_{L_\phi}\in M(X)$ such that 
\begin{equation}\label{RR}
\int_XL_\phi fd\mu=\int_X fd\mu_{L_\phi}
\end{equation}
for all $f\in C(X)$. In fact, a simple argument shows that equation (\ref{RR}) holds for any bounded $f\in L^1(\mu_{L_\phi})$. For convenience we use $L$ to denote $L_0$, and in line with Keane's original notation we use $Q\mu$ to denote $\mu_{L}$. We can now present two most common definition of a $g$ measure.

\begin{defn}\label{gmeasfixedpointdefn}
A $T$-invariant probability measure $\mu$ is called a $g$-measure if, for some $0<g\in\G_c$,
\[L_{\log g}^*\mu=\mu.\]
\end{defn}
\noindent
Equivalently,
\begin{defn}\label{defgmeas}
A probability measure $\mu$ on $(X,\mathcal B)$ is called, for some $g\in \mathcal G$, a $g$-measure if $\mu\ll Q\mu$ and $\frac{d\mu}{dQ\mu}=g$.
\end{defn}

We denote, as usual by $P(X)$ the set of Borel probability measures on $X$, and by  $P_T(X)$ the sub-collection of $T$-invariant Borel probability measures. Recall that a measure $\mu\in P_T(X)$ is an equilibrium state for $\phi\in C(X)$ if 
\[h_\mu(T)+\mu(\phi)=\sup_{\nu\in P_T(X)}\{h_\nu(T)+\nu(\phi)\}.\]
The right hand side of the above equation is often called the pressure of $T$ with respect to $\phi$, which we shall denote by $\mathcal P_T(\phi)$. We should also note that the notion of pressure and equilibrium states are still perfectly well defined if we consider bounded measurable potentials instead of just continuous ones. It then makes sense, for a given $Y\subset X,$ to define the pressure of $T$ with respect to a potential $\phi$ concentrated on $Y$ by
\[\mathcal P_T^Y(\phi):=\sup_{\nu\in P_T(X)}\{h_\nu(T)+\nu(\phi):\nu(Y)=1\}.\] We thus have the following theorem about concentrated pressures and measure isomorphisms:
\begin{thm}\label{pressureiso}
If $\mu\in P_T(X)$ is an equilibrium state for $\phi\in C(X)$, and there exists measure isomorphism $F$ from $(X,{\sigma},\B,\mu)$ to another measurable dynamical system $(Y,S,\F,F_*\mu)$ where $F:X'\to F(X')$ is a bijection, and $T(X')\subset T$, and $\mu(X')=1$ then $\mathcal P_T(\phi)=\mathcal P_S^{F(X')}(\phi_F)$, where $\phi_F=\phi\circ F^{-1}$ $F_*\mu$ a.e., and $F_*\mu$ is an equilibrium state for $\phi_F$.
\end{thm}
\begin{proof}
Note that $\phi\circ F^{-1}$ is only defined on $F(X')$, but we can extend it to a potential, say $\phi_F\in L^1(F_*\mu)$, by defining $\phi_F|_{F(X')}=\phi\circ F^{-1}$, and $\phi_F|_{Y\backslash F(X')}=0$. By our assumption,$h_\mu(T)=h_{F_*\mu}(S)$. Furthermore, noting that $F_*\mu(F(X'))=\mu(F^{-1}F(X'))=1$,
\[F_*\mu(\phi_F)=\int_{F(X')}\phi\circ F^{-1}dF_*\mu=\int_{X'}(\phi\circ F^{-1})\circ Fd\mu=\mu(\phi).\]
Now assume there exists another $S$ invariant probability measure on $Y$ supported on $X'$ , say $\nu$. Define $G:Y\to X$ by $G|_{F(X')}=F^{-1}$, and $G|_{Y\backslash F(X')}=x$ for an arbitrary $x\in X$. Clearly $G_*\nu$ is then a probability measure on $X$ with $G_*\nu(X')=1$. Thus the systems $(Y,S,\F,\nu)$ and $(X,T,\B,G_*\nu)$ are measurably isomorphic. Hence, $h_\nu(S)=h_{G_*\nu}(T)$, and 
\[G_*\nu(\phi)=\int_{F(X')}\phi\circ Gd\nu=\int_{F(X')}\phi\circ F^{-1}=\nu(\phi_F).\]
By definition we know $h_{G_*\nu}(T)+G_*\nu(\phi)\leq \mathcal P_T(\phi)$, so the above equalities mean that $h_\nu(\phi_F)+\nu(\phi_F)\leq \mathcal P_T(\phi)$, which proves the theorem.
\end{proof}

As most dynamical systems can be translated into an appropriate shift space, we now restrict our attention to $(X,\sigma)$-a topologically mixing-subshift of finite type, with $\sigma$ the left-shift operator, and $X\subset \mathcal A^{\N_0}$, where $\mathcal A$ is some finite alphabet. $X$ is a compact metrisable space when equipped with the topology generated from cylinder sets,\index{Cylinder Sets} which are sets of the form
\[[a_0,\dots,a_n]:=\{x\in X:x_i=a_i~\text{for all}~i=0,\dots,n\},\] for $a_i\in \mathcal A$ and $n\in {\N_0}$. We equip $X$ with the standard metric inducing the product topology, $d:X\times X\to \R^+$, defined by
\[d(x,y)=\frac{1}{\min\{i:x_i\neq y_i\}+1}.\]
In 1974 Ledrappier \cite{Ledrappier} found an interesting link between $g$-measures, conditional expectations and equilibrium states in this setting. 
\begin{thm}\label{Ledrappier}
For $g\in \G_c$ and $\mu\in P(X)$ the following are equivalent:
\begin{itemize}
\item[(i)] $\mu$ is a $g$-measure.
\item[(ii)] $\mu\in P_\sigma(X)$ and 
\[\Eb_\mu(f|\sigma^{-1}\B(X))(x)=\sum_{y\in \sigma^{-1}\circ \sigma x}g(y)f(y)\quad \mu~\text{a.e.}\]
for all $f\in L^1(\mu).$
\item[(iii)] $\mu\in P_\sigma(X)$ and $\mu$ is an equilibrium state for $\log g$.
\end{itemize}
\end{thm}

A useful technical lemma, which we will need later, also follows from Ledrappier's Theorem \ref{Ledrappier}.
\begin{lem}[{\cite[Lemma 2.1]{WaltRu}}]\label{gmeassupport}\ \\
\begin{itemize}
\item[(i)] If $g\in \G_c$ then every $g$-measure has full support.
\item[(ii)] If $g_1,g_2\in \G_c$ share a $g$-measure $\mu$, then $g_1=g_2$.
\end{itemize}
\end{lem}

The following equivalent definitions of $g$-measures also prove useful.
\begin{prop}\label{gmeasnewequiv}
For $g\in\G_c$ and $\mu\in P(X)$ the following are equivalent:
\begin{itemize}
\item[(i)] $\mu$ is a $g$-measure.
\item[(ii)] $\mu\in P_\sigma(X)$,and for every $f\in L^1(\mu)$.
\[\int_X fd\mu =\int_X \sum_{y\in \sigma^{-1}x}f(y)g(y)d\mu(x)\]
\item[(iii)] For every $f\in L^1(\mu)$
\[\int_X fd\mu =\int_X \sum_{y\in \sigma^{-1}\sigma x}f(y)g(y)d\mu(x).\]
\end{itemize}
\end{prop}

\bigskip
We want to find conditions on $g$-functions on $(X,\sigma)$ guaranteeing unique equilibrium states. To that end we define, for any $\phi\in C(X)$ and $n\in \N$ the $n$th-variation \index{$n$-th variation}
\[\var_n(\phi):=\sup\{|\phi(x)-\phi(y)|:y_i=x_i,\, 0\leq i\leq n-1\}.\]
We introduce the function $\sigma_n:C(X)\to C(X)$ with the definition
\[\sigma_n\phi(x)=\sum_{j=0}^{n-1}\phi(\sigma^j x).\]We say that $\phi$ satisfies the so-called Bowen condition \cite{WaltersBow} if there exists a $k\in \N$ such that 
\[\sup_{n\geq 1}\var_{n+k}(\phi\circ \sigma^n)<\infty.\] We denote by $\bow(X,\sigma)$ the space of all $\phi\in C(X)$ satisfying the Bowen condition.
Finally we say that $\phi\in C(X)$ has \textit{summable variation} if
\[\sum_{n=1}^\infty\var_n(\phi)<\infty.\]

The utility of $g$-measures arises in light of the following classic theories. In 1975 Walters \cite{WaltRu} adapted a proof of Keane's \cite{Keane} to prove the following:
\begin{thm}[{\cite[Theorem 3.1]{WaltRu}}]\label{RuelleconvergenceThm}
If $g\in \G_c$ and $\log g$ has summable variation then there exists a unique $g$ measure $\mu$, and $L_{\log g}f$ converges uniformly to $\mu(f)$ for every $f\in C(X)$.
\end{thm}


The above theorem implies that there is a strong connection between $g$-measures and the ergodic properties of $\sigma$.
\begin{thm}[{\cite[Theorem 3.2]{WaltRu}}]\label{RuelleMixing}
Let $(X,\sigma)$ be a one-side topologically mixing subshift of finite type, and $g\in \G_c$ have $\log g$ finite variation. Then the unique $g$-measure, $\mu$, given by Theorem \ref{RuelleconvergenceThm}, has a Bernoulli natural extension. Thus $\sigma$ is strongly mixing with respect to $\mu$, and is an exact endomorphism.
\end{thm}

This leads to a fundamental theorem, crucial to studying equilibrium states \cite{WaltRu}:
\begin{thm}[Ruelle's Operator Theorem {\cite[Theorem 3.3]{WaltRu}}]\label{RuelleOp}
If $\phi\in C(X)$ has summable \index{Ruelle's Operator Theorem} variation, then there exists a $\lambda>0$, strictly positive $h\in C(X)$, and $\nu\in P(X)$ such that $\nu(h)=1$, $L_\phi h=\lambda h$, $L^*_\phi\nu=\lambda\nu$, and $\|L_\phi^n f/\lambda^n-\nu(f)h\|_\infty\to 0$ for all $f\in C(X)$.
\end{thm}
The importance of Ruelle's operator theorem is made clear by the following corollary \cite{WaltRu}:
\begin{cor}[{\cite[Corollary 3.3(i)]{WaltRu}}]\label{RuelleCor}
Let $\phi\in C(X)$ have summable variation. $\phi$ has a unique equilibrium state, $\mu_\phi$, with full support, satisfying $\mu_\phi(f)=\nu(hf)$ for any $f\in C(X)$, with $\nu$ and $h$ as in Theorem \ref{RuelleOp}. Furthermore, $\mu_{\phi}$ is the unique $g$-measure of
\[g:=\frac{e^\phi h}{\lambda (h\circ \sigma)},\]
with $\lambda$ as in Theorem \ref{RuelleOp}. $\sigma$ is an exact endomorphism with respect to $\mu_\phi$, with Bernoulli natural extension. $\lambda$ is the spectral radius of $L_\phi$, and $\mathcal P_\sigma(\phi)=\log \lambda$.
\end{cor}

In a subsequent paper \cite{WaltersBow} Walters proved a strengthening of the above, for potentials satisfying the Bowen condition ($\sup_{n\geq 1}\var_{n+k}(\phi\circ \sigma^n)<\infty$ for some $k\in \N$) and not just finite summability.
\begin{thm}\label{unique2}
	If $g\in \G_c\cap \bow(X,\sigma)$ then there exists a unique $g$-measure $\mu$.
\end{thm}
The proof of Theorem \ref{unique2} relies on tools developed in an earlier paper of his, \cite{WaltersDist}, which deals with equilibrium states of more general compact metric spaces and maps on them which expand distances. In fact Theorem \ref{unique2} itself is applicable to more general spaces and maps than just subshifts of finite type. It is also actually part of a convergence theorem akin to Ruelle's Operator Theorem. In our restricted case it was actually proven earlier by Rufus Bowen in his monograph \textit{Ergodic Theory and the Ergodic Theory of Anosov Diffeomorphisms} \cite{BowenBookOrig}.

\bigskip
We end his section by stating a known result that any Markov measure is a $g$-measure (see \cite{Keane}). To be more precise,
let $(X,\sigma)$ be as above (with $|\mathcal A|=n$), and let $P$ be a transition matrix with a strictly positive left stationary distribution $\pi$. Let $Y$ be the subshift of finite type induced on $X$ by the Markov chain. By this we mean that $y\in Y$ if and only if $P_{y_i,y_{i+1}}>0$ for all $i\in \N_0$. Let $\mu$ be the Markov measure on $Y$ generated by $(P,\pi)$. We know that $\mu$ is $\sigma$-invariant. As Keane shows in \cite{Keane} we can explicitly find the $g$-function corresponding to $\mu$:
\begin{prop}\label{MarkG}
	The function $g_P:Y\to \R$ defined by
	\[g_P(x):=\frac{\pi_{x_0}P_{x_0,x_1}{\tiny {\tiny }}}{\pi_{x_1}}=\frac{\mu([x_0,x_1])}{\mu([x_1])}.\] is an element of $\G_c$, and $\mu$ is a $g$-measure corresponding to the $g$ function $g_P$.
	
\end{prop}

\section{Random $\beta$-transformations and Markov partitions}\label{MarkovChain}
In this section, we examine the collection of $\beta$'s for which the dynamics of the $K_{\beta}$ transformation can be described by a topological Markov chain. To this end, let $S=\bigcup_{1 \le k \le \lbetar} \interior (S_k)$, where $\interior(S_k)$ denotes the interior of the interval $S_k$.
\begin{defn}\label{d:B}
Let $B \subset (1,\infty)$ be the set of non-integers $ \beta >1$ with the following two properties:
\begin{itemize}
\item[(B1)] the set
\[ F=\Big\{\pi \Big(K^n\Big(\omega, \frac{k}{\beta}\Big)\Big), \pi \Big(K^n\Big(\omega, \frac{\lfloor \beta \rfloor}{\beta(\beta-1)} + \frac{k}{\beta}\Big)\Big) \, : \, k \in \{0, \ldots, \lbetar \}, \, n \ge 0,\, \omega \in \Omega\Big\}\]
is finite, and,
\item[(B2)] $F \cap S = \emptyset$.
\end{itemize}
\end{defn}

So the values of $\beta$ that satisfy the conditions of Definition~\ref{d:B} are those values for which the orbits of all the endpoints of the intervals $E_k$ and $S_k$ are ultimately periodic and do not enter $S$ under iterations of the random map $K_{\beta}$ with any possible driving sequence $\omega$.  In \cite{KarmadeVries} it was shown that the dynamics of $K_{\beta}$ can be described by a topological Markov chain, which we will quickly summarise, and refer the reader to \cite{KarmadeVries, KarmaInv} for further details.
We start by considering the partition of
\[\mathcal E=\{E_0,S_1,E_1,\dots S_{\lfloor\beta\rfloor},E_{\lfloor\beta\rfloor}\}\]
of $I_{\beta}$. We refine $\mathcal E$ with the orbits of $1$ and $\frac{\lfloor\beta\rfloor}{\beta-1}-1$ under $T_\beta$, i.e. by elements of $F$. We arrange the endpoints of each element of $\mathcal E$ along with $T_\beta^i 1$ and $T_\beta^i(\lfloor\beta\rfloor/(\beta-1)-1)$ in ascending order, and then form a new partition 
\[\mathcal C:=\{C_0,\dots, C_K\}.\] 
with the intervals determined by these endpoints. To decide the openness of these intervals, we choose them to satisfy, for all $i\in\{0,\dots,n-2\}$:
\begin{itemize}
	\item $T_\beta^i1\in C_j$ if and only if $T_\beta^i1$ is a left endpoint of $C_j$.
	\item $T_\beta^i(\frac{\lfloor\beta\rfloor}{\beta-1}-1)\in C_j$ if and only if $T_\beta^i(\frac{\lfloor\beta\rfloor}{\beta-1}-1)$ is a right endpoint of $C_j$.
\end{itemize}
The partition ${\mathcal C}$ has the following properties, which we list without proof.
\begin{prop}\label{partitionC}
	The following properties hold for the partition $\mathcal C$:
	\begin{itemize}
		\item[(i)] $C_0=[0,\lfloor\beta\rfloor/(\beta-1)-1]$ and $C_K=[1,\lfloor\beta\rfloor/(\beta-1)]$.
		\item[(ii)] We can write, for any $i\in\{0,\dots,\lfloor\beta\rfloor\}$, $E_i=\bigcup_{j\in M_i} C_j$ where $\{M_0,\dots M_{\lfloor\beta\rfloor}\}$ is a disjoint collection of subsets of $\{0,\dots,K\}$. Furthermore $|M_{i}|=|M_{\lfloor\beta\rfloor-i}|$.
		\item[(iii)] To each switch-region $S_i$ there corresponds a single $s_i\in\{0,\dots,K\}\backslash\bigcup_{k=0}^{\lfloor\beta\rfloor}M_k$ such that $S_i=C_{s_i}$.
		\item[(iv)] If, for some $i,j$, $C_j\subset E_i$ then $T_\beta(C_j)=S_\beta(C_j)=\bigcup_{j=1}^lC_{i_j}$ for some $l\in\{0,\dots, K\}$. Furthermore  $T_\beta(C_{K-j})=\bigcup_{i=1}^lC_{K-j_i}$.
		\item[(v)] If $C_j=S_i$ for some $i,j$, then  $T_\beta (C_j)=C_0$ and $S_\beta(C_j)=C_K$.
	\end{itemize} 
\end{prop}

\bigskip
On $\Omega\times I_{\beta}$ we consider the partition 
\[\mathcal P:=\left\{\Omega\times C_j:j\in\bigcup_{k=0}^{\lfloor\beta\rfloor}M_k\right\}\bigcup\{\{\omega_0=i\}\times S_j:i\in\{0,1\},~j\in\{1,\dots \lfloor\beta\rfloor\}\}.\]
We can now define the $K+1\times K+1$ adjacency matrix $A=(a_{i,j})$ defining the subshift of finite type underlying \index{Markov Partition}$K_\beta$:
\begin{equation}
a_{i,j}:=\begin{cases}1 & \text{if}~i\in\bigcup_{k=0}^{\lfloor\beta\rfloor} M_k~\text{and}~C_j\cap T_\beta C_i=C_j \\
0 & \text{if}~i\in\bigcup_{k=0}^{\lfloor\beta\rfloor} M_k~\text{and}~C_i\cap T_\beta^{-1} C_j=\emptyset\\
1 & \text{if}~i\in\{0,\dots,K\}\backslash\in\bigcup_{k=0}^{\lfloor\beta\rfloor} M_k~\text{and}~j\in\{0,K\}\label{adjacencymatrix}\\
0 & \text{if}~i\in\{0,\dots,K\}\backslash\in\bigcup_{k=0}^{\lfloor\beta\rfloor} M_k~\text{and}~j\notin\{0,K\}.\end{cases}
\end{equation}
Let $Y$ be the subshift of finite type determined by $A$. $(Y,\sigma)$ is topologically mixing, because $A$ is irreducible, and there is always a positive entry on the main diagonal, namely $a_{0,0}=1$, which means that $A$ is aperiodic. An irreducible, aperiodic matrix is primitive and a subshift of finite type with a primitive matrix is topologically mixing. We want to define a map $\psi:Y\to \Omega\times I_{\beta}$, so that the dynamics of $K_\beta$ are essentially represented by the dynamics of the left shift on $Y$, $\sigma_Y$. For convenience we denote by $s_i$ the states corresponding to the switch regions $S_i$, for $i\in \{1,\dots,\lfloor\beta\rfloor\}$. We first define how $\psi$ maps onto the second coordinate: To each $y\in Y$ we associate $(e_i(y))\in\{0,\dots,\lfloor\beta\rfloor\}^\N$ given by
\begin{equation}
e_i(y)=\begin{cases}j&\text{if}~y_i\in M_j\\
j&\text{if}~y_i=s_j~\text{and}~y_{i+1}=0\\
j-1&\text{if}~y_i=s_j~\text{and}~y_{i+1}=K.
\end{cases}\label{ytoemarkov}
\end{equation}
We can now associate an $x_y$ to $y$ by setting
\begin{equation}
x_y:=\sum_{i=1}^\infty\frac{e_i(y)}{\beta^i}\label{xy}.
\end{equation}
Unfortunately we cannot associate a unique $\omega_y$ to $y\in Y$ if $y$ does not contain infinitely many entries of the form $s_i$. For that reason we define 
\[Y':=\{(y_i)\in Y:y_i\in\{s_1,\dots,s_{b_1}\}~\text{for infinitely many}~i'\text{s}\}.\label{pageY'}\]
For $y\in Y'$ we define $n_i(y)$ to be the $i$-th time that $y_j\in\{s_1,\dots s_{\lfloor\beta\rfloor}\}$ for some $j\in\N$. That is $y_j\in \{s_1,\dots,s_{\lfloor\beta\rfloor}\}$ if and only if $j=n_i(y)$ for some $i\in \N$, and $n_i(y)<n_{i+1}(y)$ for all $i\in \N$. With this we can associate a $\omega^y\in\Omega$ with $y$ by
\begin{equation}\label{omegay}
\omega_i^y=\begin{cases}1&\text{if}~y_{n_i(y)+1}=0\\
0& \text{if}~y_{n_i(y)+1}=K.\end{cases}
\end{equation}
Note that this is well defined by Proposition \ref{partitionC} $(v)$. Finally we can define $\psi:Y'\to \Omega\times[0,\lfloor\beta\rfloor/(\beta-1)]$ by $\psi(y)=(\omega^y,x_y)$, with $x_y$ and $\omega^y$ defined by \eqref{xy} and \eqref{omegay} respectively. We can extend the domain of $\psi$ to $Y$ by defining, for $y\in Y\backslash Y'$, $\psi(y)=(\omega,x_y)$, where $\omega_i$ is defined as in definition \ref{omegay} for all $i$ where $n_i(y)$ is defined (this will be just be the first finite number of entries). For the other $i$ we set $\omega_i=1$. The set $Y'$ has the following properties (see \cite{KarmadeVries}).
\begin{lem}\label{Markovdynlem1}
For $y'\in Y$ the following hold:
\begin{itemize}
\item[(i)] If $y_1=k$ for $k\in \bigcup_{i=0}^{\lfloor\beta\rfloor}M_i$ then $x_y\in C_k$.
\item[(ii)]  If, for some $i\in\{1,\dots,\lfloor\beta\rfloor\}$, $y_1=s_i$, $y_2=0$ then $x_y\in S_i$ and $\omega_1=1$
\item[(iii)]  If, for some $i\in\{1,\dots,\lfloor\beta\rfloor\}$, $y_1=s_i$, $y_2=K$ then $x_y\in S_i$ and $\omega_1=0$.
\item[(iv)] $\psi$ conjugates $K_\beta$ and $\sigma_Y$ restricted to $Y'$. That is for any $y\in Y'$
\[K_\beta\circ \psi(y)=\psi\circ \sigma_Y(y).\]	
\end{itemize}	
\end{lem}

This allows us to conclude this section with an isomorphism theorem, the proof of which is a simple adaptation of Theorem 4 in \cite{KarmadeVries}. Recall that $\mathcal B_1\times \B_2$ is the product Borel $\sigma$ algebra on $\Omega\times I_{\beta}$, and let $\F$ be the Borel (which is also the product) $\sigma$ algebra on $Y$.
\begin{thm}\label{measureiso}
If $\mu\in P_\sigma(Y)$ satisfies $\mu(Y')=1$, then the dynamical systems\\ $(\Omega\times I_{\beta},\B_1\times B_2, \mu_*\psi,K_\beta)$ and $(Y,\F, \mu,\sigma)$ are measurably isomorphic.
\end{thm}

\section{$g$-measures for the random $\beta$-transformations}\label{gmeasures}
As a motivation for the choice of potentials that we will be considering in this section, we start with the Dyson model. Consider $X=\{-1,1\}^{\N_0}$ with $\sigma$ the left shift on $X$. This space is of great importance for thermodynamic formalism, as it can be used to model lattices of particles with opposite parity \cite{RuelleBook}, \cite{Hans}. In particular it serves as the underlying structure for Dyson models of Ising spins \cite{IsingBissacot}\index{Dyson Models}. The long-range interaction potential $\varphi_{\alpha}\in C(X)$ for $1<\alpha<2$ and $J\geq 0$ defined by 
\[\varphi_{\alpha,J}(x):=-J\sum_{\substack{(i,j)\in \N_0^2\\ i\neq j}}\frac{x_ix_j}{|i-j|^\alpha}\]
is useful in the study of ferromagnetic lattices. Bissacot et. al. \cite{IsingBissacot} show that under certain conditions there exists so called Gibbs measures for $\varphi_{\alpha,J}$ which are not $g$-measures. The theory of Gibbs measures, even formally defining them, is out of the scope of this paper, and we direct the interested reader to \cite{RuelleBook}.  Clearly $\varphi_{\alpha,J}$ does not have summable variation, so the power of the Ruelle Operator Theorem can not be brought to bear. The investigation into the conditions under which $\varphi_{\alpha,J}$ has a unique $g$-measure is still ongoing, although Pollicott et.a al. have shown that there exists non-unique $g$-measures under certain conditions \cite{PollicottIsing1}. Closely related are the local versions of such potentials: 
\[\varphi_{\alpha,J}^0(x):=-J\sum_{n=1}^\infty\frac{x_0x_n}{n^\alpha}.\]
If $J>0$ we can use the standard integral comparison test from real analysis to calculate, for $1<\alpha<2$:
\[\sum_{n=1}^\infty\var_n(\psi_{\alpha,J}^0)\geq2J\sum_{n=1}\int_n^\infty\frac{1}{x^\alpha}=\frac{2J}{\alpha-1}\sum_{n=1}^\infty\frac{1}{n^{\alpha-1}},\]
which clearly diverges. Thus such $\varphi_{\alpha,J}^0$ do not have summable variation. The existence of unique $g$-measures for $\varphi_{\alpha,J}^0$ is still an open problem. The above calculation does show that if $\alpha>2$ then $\varphi_{\alpha,J}^0$ does have summable variation and so does possess a unique $g$-measure.\\

Defining the potentials $\varphi_{\alpha,J}^0$ as we have was for purely physical reasons, of which have little relevance to our study of random $\beta$-Transformations. We are thus motivated to
define the potential $\varphi\in C(X)$ by
\[\varphi(x)=\sum_{n=0}^\infty\frac{x_n}{2^n}.\]
Clearly $\varphi$ is continuous, because if $d(x,y)<\frac{1}{k}$ then \[|\varphi(x)-\varphi(y)|\leq\sum_{n={k+1}}2^{-k}=\frac{1}{2^k}\sum_{n=0}^\infty2^{-n}=\frac{1}{2^{k-1}}.\]
This argument shows that $\var_n(\varphi)=2^{-n+1}$, and so $\varphi$ has summable variation. In light of Ruelle's Operator Theorem \ref{RuelleOp} thus gives us the existence of some $h\in C(X)$, $\nu\in P(X)$ satisfying the properties of the theorem. For this particular $\varphi$ it is possible to explicitly construct the $h$:

\begin{lem}
	The function $h\in C(X)$ defined by
	\[h(x)=e^{\varphi(x)}\]
	is an eigenfunction of $L_\varphi$ with eigenvalue $e^2+e^{-2}$.
\end{lem}
\begin{proof}
	For any $x\in X$ we have
	\begin{align*}
		L_\varphi h(x)& =\sum_{y\in \sigma^{-1}}e^{\varphi(y)}e^{\varphi(y)}\\
		& =e^{\varphi(-1,x_0,x_1,\dots)}e^{\varphi(-1,x_0,x_1,\dots)}+e^{\varphi(1,x_0,x_1,\dots)}e^{\varphi(1,x_0,x_1,\dots)}\\
		& = e^{-2+\sum_{n=0}^\infty \frac{x_n}{2^n}}+e^{2+\sum_{n=0}^\infty \frac{x_n}{2^n}}\\
		& =(e^2+e^{-2})e^{\varphi(x)}\\
		& = (e^2+e^{-2})h(x).
	\end{align*}
\end{proof}
The fact that we are able to so easily construct a potential with summable variation, and construct an eigenfunction, motivates the rest of this section.

Let $\beta\in B$ and let $Y$ be the corresponding subshift of finite type as described in Section \ref{MarkovChain}. However, instead of $Y=\{0,\dots,K\}^\N$ we redefine $Y$ to equal $\{0,\dots,K\}^{\N_0}$ in order to be consistent with the machinery developed in section \ref{gMTO}. There is a natural homeomorphism between $\{0,\dots,K\}^\N$ and $\{0,\dots,K\}^{\N_0}$ which preserves the action of $\sigma$, so this new definition in no way affects $Y$'s role as a Markov partition of $(\{0,1\}^\N\times I_{\beta},K_\beta)$. To each $y\in Y$ we associate an $(e_i(y))\in\{0,\dots, \lfloor\beta\rfloor\}^{\N_0}$ with the appropriate modification of formula \eqref{ytoemarkov} (i.e. just shifting the coordinates by one). Recall that we equip $Y$ with the standard product topology, which is generated by cylinder sets, and is compatible with the metric
\[d(x,y)=\frac{1}{1+\min\{k:y_k\neq x_k\}}.\]
Let us define $\gamma:Y\to \{0,\dots,b_1\}^\N$ by $\gamma(y)=(e_i(y))$. Let us define a potential $\varphi\in C(Y)$ by
\[\varphi(y)=\sum_{n=0}^\infty \frac{e_n(y)}{2^n}.\]
Before discussing $\varphi$ further, we need the following lemma:
\begin{lem}\label{shiftpreimage}
	For any $y=(y_0,y_1\dots)\in Y$: \[\gamma(\sigma^{-1}(y))=\{(i,e_0(y),e_1(y),\dots):i=0,\dots,\lfloor{\beta}\rfloor\}\]
\end{lem}
\begin{proof}
	$\sigma^{-1}(y)$ consists of all $z\in Y$ of the form $(i,y_0,\dots)$, such that $a_{i,y_0}=1$, where $A$ is the adjacency matrix \eqref{adjacencymatrix}. Now $a_{i,y_0}=1$ if and only if $i\in\bigcup_{k=0}^{\lfloor{\beta}\rfloor}M_k$ and $C_{y_0}\subset T_\beta C_i$ or $i\in \{0,\cdots,K\}\backslash \bigcup_{k=0}^{\lfloor{\beta}\rfloor}M_k$ and $y_0\in \{0,K\}$. If $y_0\in \{0,K\}$ then $i$ can be any value associated to a switch region, so each $i\in \{s_1,\dots,s_{b_1}\}$ defines an element $(i,y_0,\dots)\in\sigma^{-1}y$, so by definition \eqref{ytoemarkov} \[\gamma(\sigma^{-1}(y))\supset \begin{cases}\{1,\dots,\lfloor\beta\rfloor\}&\text{if}~y_0=0\\\{0,\dots,\lfloor\beta\rfloor-1\}&\text{if}~y_0=K.
	\end{cases}.\]
Furthermore we know $a_{0,0}=1$, because $C_0\subset T_\beta C_0$ by Proposition \ref{partitionC} \textit{(i)}. Thus if $y_0=0$ we conclude that $\gamma(\sigma^{-1}(y))=\{0,\dots,\lfloor\beta\rfloor\}$. Similarly $a_{L,L}=1$, because $C_K\subset T_\beta C_K$, from which we conclude that $\gamma(\sigma^{-1}(y))=\{0,\dots,\lfloor\beta\rfloor\}$ if $y_0=K$. If $y_0\notin\{0,K\}$ then $a_{i,y_0}$ can only equal $1$ if $i\in\bigcup_{k=0}^{\lfloor{\beta}\rfloor}M_k$ and $C_{y_0}\subset T_\beta C_i$. As $y_0\neq K$ then $C_{y_0}\neq [1,\lfloor{\beta}\rfloor/(\beta-1)]$, so $T^{-1}_\beta C_{y_0}$ consists of exactly $\lfloor{\beta}\rfloor+1$ disconnected intervals, and from Proposition \ref{partitionC}, because $C_{y_0}\neq C_0,C_K$ we know none of these intervals contain switch regions, so from the same proposition we know that they must be intervals of the form $C_i$ for $i\in\bigcup_{k=0}^{\lfloor{\beta}\rfloor}M_k$, and no two of the $C_i$ can be have their index from the same $M_k$, as they will then be contained in the same equality region, on which $T_\beta$ is injective. Thus from \eqref{ytoemarkov} we can conclude that $\gamma(\sigma^{-1}(y))=\{0,\dots,\lfloor\beta\rfloor\}$.
\end{proof}

With this Lemma we can prove the following:
\begin{prop}
The function $H\in C(Y)$ defined by
$H(y)=e^{\varphi(y)}$ is an eigenfunction of $L_\varphi$, with eigenvalue $\frac{e^{2\lfloor{\beta}\rfloor+2}-1}{e^2-1}$.
\end{prop}
\begin{proof}
For any $y\in Y$:
\begin{align*}
L_\varphi H(y)&= \sum_{z\in\sigma^{-1}(y)}e^{2\varphi(z)}\\
& = \sum_{i=0}^{\lfloor{\beta}\rfloor}e^{2\varphi(i,y_0,\dots)},
\end{align*}
where the last equality follows from lemma \ref{shiftpreimage}.
Thus
\begin{align*}
L_\varphi H(y)
& = \sum_{i=0}^{\lfloor{\beta}\rfloor}e^{2i+2\sum_{j=1}^\infty\frac{e_{j-1}(y)}{2^j}}\\
& =\sum_{i=0}^{\lfloor{\beta}\rfloor}e^{2i+\sum_{j=0}^\infty\frac{e_{j}(y)}{2^{j}}}\\
& =e^{\varphi(y)}\sum_{i=0}^{\lfloor{\beta}\rfloor}e^{2i}\\
& = \frac{e^{2\lfloor{\beta}\rfloor+2}-1}{e^2-1}H(y)
\end{align*}
\end{proof}
The ease of the above calculations suggests an approach of generating uncountably infinitely many unique equilibrium states for $\sigma$ on $Y$. Fix an arbitrary function $\theta:\{0,\dots,\lfloor{\beta}\rfloor\}\to \R$. We claim there exists a potential $\varphi_\theta\in C(Y)$ of summable variation and a corresponding eigenfunction of $L_{\varphi_{\theta}}$, say $H\in C(Y)$, corresponding to the eigenvalue $\sum_{i=0}^{\lfloor{\beta}\rfloor}e^{\theta(i)}$.
\begin{prop}\label{infiniteeigenvalues}
Define $\varphi_{\theta}, H\in C(Y)$ by
\[\varphi_\theta(y)=\sum_{j=0}^\infty\frac{\theta(e_j(y))}{2^{j}},\quad H(y)=e^{\varphi_\theta(y)}.\]
$H$ is an eigenvector of $L_{\varphi_\theta}$ corresponding to eigenvalue $\sum_{i=0}^{\lfloor{\beta}\rfloor}e^{2\theta(i)}$
\end{prop}
\begin{proof}
For convenience we set $\varphi:=\varphi_{\theta}$. Due to Lemma \ref{shiftpreimage}, for any $y\in Y$,
\begin{align*}
L_\varphi H(y)&=\sum_{x\in \sigma^{-1}y}e^{2\sum_{j=0}^\infty \frac{\theta(e_j(x))}{2^j}}\\
& =\sum_{i=0}^{\lfloor{\beta}\rfloor}e^{2\theta(i)+2\sum_{j=0}^\infty \frac{\theta(e_j(y))}{2^{j+1}}}\\
& =H(y)\sum_{i=0}^{\lfloor{\beta}\rfloor}e^{2\theta(i)}.	
\end{align*}
\end{proof}
Although motivated by the Dyson model for the long range interaction of dipoles, the potentials $\varphi_{\theta}$ have a glaring difference in that their dependence of the tail decreases exponentially, instead of at a polynomial rate. This is necessary to ensure that the potentials have summable variation, allowing us to bring Ruelle's Operator Theorem to bear. Although we can not define a potential with strict polynomial decay on the terms, the following proposition shows we can get arbitrarily close in some sense.

\begin{prop}\label{infiniteeigenvaluesalpha}
Fix some $\alpha>1$ and $\theta:\{0,\dots,\lfloor{\beta}\rfloor\}\to\R$. Define $\varphi_{\theta,\alpha},H_{\theta,\alpha}\in C(Y)$ by 
\[\varphi_{\theta,\alpha}(y)=\sum_{j=0}^\infty\frac{\theta(e_j(y))(\alpha-1)^{j}}{\alpha^{j+1}},\quad H_{\theta,\alpha}(y)=e^{(\alpha-1)\varphi_{\theta,\alpha}(y)}.\]
Then $\sum_{i=0}^{\lfloor{\beta}\rfloor}e^{\theta(i)}$ is an eigenvalue of $L_{\varphi_{\theta,\alpha}}$ with eigenvector $H_{\theta,\alpha}$.
\end{prop}
\begin{proof}
For any $y\in Y$ we have, by Lemma \ref{shiftpreimage}
\begin{align*}
L_{\varphi_{\theta,\alpha}} H_{\theta,\alpha}(y)&=\sum_{x\in \sigma^{-1}y}e^{(\alpha-1)\varphi_{\theta,\alpha}(x)}e^{\varphi_{\theta,\alpha}(x)}\\
& =\sum_{i=0}^{\lfloor{\beta}\rfloor}e^{\alpha\varphi_{\theta,\alpha}(i,y_0,\dots)}\\
&=\sum_{i=0}^{\lfloor{\beta}\rfloor}e^{\alpha\Big(\sum_{j=1}^\infty\frac{\theta(e_{j-1}(y))(\alpha-1)^{j}}{\alpha^{j+1}} +\frac{\theta(i)}{\alpha}\Big)}\\
& = e^{\sum_{j=0}^\infty\frac{\theta(e_{j}(y))(\alpha-1)^{j+1}}{\alpha^{j+1}}}\sum_{i=0}^{\lfloor{\beta}\rfloor}e^{\theta(i)}\\
& =e^{(\alpha-1)\varphi_{\theta,\alpha}(y)}\sum_{i=0}^{\lfloor{\beta}\rfloor}e^{\theta(i)}\\
& =\sum_{i=0}^{\lfloor{\beta}\rfloor}e^{\theta(i)} H_{\theta,\alpha}(y).
\end{align*}
\end{proof}
To apply Ruelle's Operator Theorem we first need to check whether $\varphi_{\theta,\alpha}$ has summable variation:
\begin{prop}
$\varphi_{\theta,\alpha}$ possesses summable variation for any $\alpha>1$ and $\theta:\{0,\dots,\lfloor{\beta}\rfloor\}\to\R$.
\end{prop}
\begin{proof}
By definition we have for every $n\in \N$,
\[\var_n(\varphi_{\theta,\alpha})\leq2\|\theta\|_\infty\sum_{j=n}^\infty\frac{(\alpha-1)^j}{\alpha^{j+1}}\leq2\|\theta\|_\infty \sum_{j=n}^\infty\frac{(\alpha-1)^j}{\alpha^{j}}\]
As $(\alpha-1)^j/\alpha^j$ is strictly decreasing in $j$, we can apply the integral test to bound, for any $n\in \N$
\begin{align*}
\sum_{j=n}^\infty\frac{(\alpha-1)^j}{\alpha^{j}}&\leq \frac{(\alpha-1)^n}{\alpha^{n}}+\int_n^\infty\frac{(\alpha-1)^x}{\alpha^{x}}dx\\
& =\frac{(\alpha-1)^n}{\alpha^{n}}+\Big(\ln(\frac{\alpha}{\alpha-1})\Big)^{-1}e^{-n\ln(\frac{\alpha}{\alpha-1})}
\end{align*}
Thus
\begin{align*}
\sum_{n=1}^\infty\var_n(\varphi_{\theta,\alpha})&\leq \sum_{n=1}^\infty\frac{(\alpha-1)^n}{\alpha^{n}}+\Big(\ln(\frac{\alpha}{\alpha-1})\Big)^{-1}e^{-n\ln(\frac{\alpha}{\alpha-1})}\\
& <\infty.
\end{align*}
\end{proof}

The natural question, and indeed the desired conclusion, from finding uncountably many potentials with different eigenvalues for the corresponding transfer operator, is whether we can find uncountably many associated unique $g$-measures. To that end, we would like to be able to bring to bear the machinery developed in Section \ref{gMTO}. 
\begin{thm}\label{eigenequality}
Fix a $\varphi_{\theta,\alpha}\in C(Y)$ as in Proposition \ref{infiniteeigenvaluesalpha}. The eigenvalue in that proposition is the same eigenvalue guaranteed by Ruelle's Operator Theorem \ref{RuelleOp} and the eigenvector $H_{\theta,\alpha}$ is, up to a normalising constant, the same as the eigenvector $h_{\theta,\alpha}$ guaranteed by Ruelle's Operator Theorem.
\end{thm}
\begin{proof}
Let us denote the eigenvalue corresponding to $H_{\theta,\alpha}$ by $\lambda_{H_{\theta,\alpha}}$, and the one coming from Ruelle's Operator Theorem by $\lambda$. Let $\nu_{\theta,\alpha}$ be the probability measure coming from Ruelle's Operator Theorem.  As $H_{\theta,\alpha}\in C(Y)$, Ruelle's Operator Theorem gives that $\|L_{\varphi_{\theta,\alpha}}^n H_{\theta,\alpha}/\lambda^n-\nu_{\theta,\alpha}(H_{\theta,\alpha})h_{\theta,\alpha}\|_\infty\to 0$, meaning that $(\lambda_{H_{\theta,\alpha}}/\lambda)^nH_{\theta,\alpha}$ converges uniformly to $\nu_{\theta,\alpha}(H_{\theta,\alpha})h_{\theta,\alpha}$. As $\nu_{\theta,\alpha}(H_{\theta,\alpha})h_{\theta,\alpha}>0$, this is only possible if $\lambda_{H_{\theta,\alpha}}=\lambda$. This immediately implies that $h_{\theta,\alpha}=H_{\theta,\alpha}/\nu_{\theta,\alpha}(H_{\theta,\alpha})$.
\end{proof}
\begin{rem}\label{remuniquenesseigen}
The reader will note that nowhere in the above proof did we use any intrinsic property of $\varphi_{\theta,\alpha}$, or of $H_{\theta,\alpha}$ or $\lambda_{H_{\theta,\alpha}}$. Thus the above theorem is valid for any potential $\phi$ with summable variation, and any eigenvalue and normalized (with respect to $\nu$) eigenvector one can find for the associated transfer operator $L_\phi$.
\end{rem}

For any $\varphi_{\theta,\alpha}$ Corollary \ref{RuelleCor} guarantees the existence of a unique equilibrium state $\mu_{\theta,\alpha}$, which is the $g$ measure corresponding to the $g$-function
\[g_{\theta,\alpha}:=\frac{e^{\varphi_{\theta,\alpha}}h_{\theta,\alpha}}{\lambda_{H_{\theta,\alpha}} (h_{\theta,\alpha}\circ\sigma)}.\]
Theorem \ref{eigenequality} immediately gives the equality
\begin{equation}\label{eqgalphatheta}
g_{\theta,\alpha}:=\frac{e^{\varphi_{\theta,\alpha}}H_{\theta,\alpha}}{\lambda_{H_{\theta,\alpha}} (H_{\theta,\alpha}\circ\sigma)}.\
\end{equation}
This allows us to prove the following uniqueness criterion.
\begin{thm}\label{uniquepot}
Two potentials $\varphi_{\theta,\alpha},\varphi_{\theta',\alpha'}\in C(Y)$ have the same unique equilibrium state if and only if $\theta=\theta'$.
\end{thm}
\begin{proof}
We know via Corollary \ref{RuelleCor} that $\varphi_{\theta,\alpha}$ and $\varphi_{\theta',\alpha'}$ have unique equilibrium states $\mu_{\theta,\alpha}$ and $\mu_{\theta',\alpha'}$ respectively, which are $g$-measures of the $g$-functions $g_{\theta,\alpha}$ and $g_{\theta',\alpha'}$ as in equation \eqref{eqgalphatheta}. If $\mu_{\theta,\alpha}=\mu_{\theta',\alpha'}$ then $g_{\theta,\alpha}=g_{\theta',\alpha'}$ by Lemma \ref{gmeassupport}. We now calculate, for any $y=(y_0,y_1,\dots)\in Y$,
\begin{align*}
g_{\theta,\alpha}(y)&=\frac{e^{\varphi_{\theta,\alpha}(y)}e^{(\alpha-1)\varphi_{\theta,\alpha}(y)}}{\lambda_{H_{\theta,\alpha}} e^{(\alpha-1)\varphi_{\theta,\alpha}(Ty)}}\\
& =\frac{1}{\lambda_{H_{\theta,\alpha}}}e^{\alpha\varphi_{\theta,\alpha}(y)} e^{(1-\alpha)\varphi_{\theta,\alpha}(Ty)}\\
& =\frac{1}{\lambda_{H_{\theta,\alpha}}}e^{\Big(\sum_{j=0}^\infty\frac{\theta(e_j(y))(\alpha-1)^j}{\alpha^j}-(\alpha-1)\sum_{j=0}^\infty\frac{\theta(e_{j+1}(y))(\alpha-1)^j}{\alpha^{j+1}}\Big)}\\
& =\frac{1}{\lambda_{H_{\theta,\alpha}}}e^{\Big(\theta(e_0(y))+\sum_{j=1}^\infty\frac{\theta(e_j(y))(\alpha-1)^j}{\alpha^j}-\sum_{j=1}^\infty\frac{\theta(e_{j}(y))(\alpha-1)^{j}}{\alpha^{j}}\Big)}\\
& =\frac{e^{\theta(e_0(y))}}{\sum_{i=0}^{\lfloor{\beta}\rfloor} e^{\theta(i)}}.
\end{align*}
This calculation holds for any $\alpha>1$ and $\theta:\{0,\dots,\lfloor{\beta}\rfloor\}\to \R$, and any $y\in Y$. Thus $g_{\theta,\alpha}=g_{\theta',\alpha'}$ if and only if $\theta=\theta'$. This also proves the converse, for if $\theta=\theta'$ we have $g_{\theta,\alpha}=g_{\theta',\alpha'}$, and because $\mu_{\theta,\alpha}$ and $\mu_{\theta',\alpha'}$ are the unique $g$-measures for $g_{\theta,\alpha}$ and $g_{\theta',\alpha'}$ they must indeed be equal.
\end{proof}

\subsection{The Importance of Explicit Eigenfunctions}

We have now calculated an uncountable number of unique $g$-measures on $(Y,\sigma)$ by constructing explicit potentials in $C(Y)$. Although the measures we construct from Ruelle's Operator Theorem arise from the Schauder-Tychonoff fixed point theorem, the fact that we have calculated explicit eigenfunctions in conjunction with Theorem \ref{eigenequality} allows us to deduce some of the measures' behaviour. In particular we know for any $\varphi_{\theta,\alpha}$ with equilibrium state $\mu_{\theta,\alpha}$, eigenfunction $H_{\theta,\alpha}$ and corresponding measure $\nu_{\theta,\alpha}$, which from now on we denote by $\varphi$, $\mu$, $H$ and $\nu$ respectively, that, due to Corollary \ref{RuelleCor} and 

Theorem \ref{eigenequality}, for any bounded measurable $f$ we can calculate:
\begin{align*}
\mu(f)&=\frac{1}{\nu(H)}\int_YfHd\nu\\
& =\frac{1}{\lambda\nu(H)}\int_Y L_\varphi(fH)d\nu\\
& =\frac{1}{\lambda\nu(H)}\int_Y\sum_{x\in\sigma^{-1}x}e^{\alpha\phi(x)}f(x)d\nu(y).
\end{align*}
In particular for any set $A$ possessing the property, that for any $a\in A$ we have $\sigma^{-1}a=\bigcup_{i=0}^{\lfloor{\beta}\rfloor}(c_i,a_0,\dots)$ for a fixed set $\{c_i\}_{i=0}^{\lfloor{\beta}\rfloor}\subset \{0,\dots,K\}$, we have

\begin{align*}
\mu(A) &=\frac{1}{\lambda\nu(H)}\sum_{i=0}^{\lfloor\beta\rfloor}\int_Y e^{\theta(i)+\sum_{j=0}^{\infty}\theta(e_j(y))\alpha^{-j-1}(\alpha-1)^{j+1}}\I_A(c_i,y)d\nu(y)\\
&=\frac{1}{\lambda\nu(H)}\sum_{i=0}^{\lfloor\beta\rfloor}e^{\theta(i)}\int_Y e^{(\alpha-1)\varphi(y)}\I_A(c_i,y)d\nu(y)\\
& =\frac{1}{\lambda\nu(H)}\sum_{i=0}^{\lfloor\beta\rfloor}e^{\theta(i)}\int_Y H(y)\I_A(c_i,y)d\nu(y)\\
& =\frac{1}{\lambda\nu(H)}\sum_{i=0}^{\lfloor\beta\rfloor}e^{\theta(i)}\mu(A\cap[c_i]),
\end{align*}
where in the last equality we applied $\sigma$ invariance of $\mu$.

\section{$K_\beta$-Invariant Measures}
Theorem \ref{uniquepot} shows that there exists uncountably many distinct $\sigma$-invariant measures on $Y$ with full support. The next question to ask is how do these measures relate to the dynamics of $K_\beta$? On page \pageref{pageY'} we followed the construction in \cite{KarmadeVries}, and created a set $Y'\subset Y$, which we identified with $\Omega\times I_{\beta}$ via the bijection $\psi$. Unfortunately this map is not in general continuous, so pulling back continuous potentials is generally not possible via this map. However, if we can show that $\mu_{\theta,\alpha}(Y')=1$ then $\psi$ is a measurable isomorphism between $(Y,\F, \mu_{\theta,\alpha},\sigma)$ and $(\Omega\times I_{\beta},\B_1\times B_2, \psi_*\mu_{\theta,\alpha},K_\beta)$.
\begin{thm}
	$\mu(Y')=1$ for any $g$-measure $\mu$ with $g\in \mathcal G_c$.
\end{thm}
\begin{proof}
We first note by Theorem \ref{gmeassupport} that $\mu$ is strictly positive on open sets, and also by Theorem \ref{RuelleMixing} that $\mu$ is strongly mixing on $(Y,\sigma)$ which means that it is also ergodic. For any switch symbol $s_i$ set 
\[Y_{s_i}=\{(y_n)\in Y:y_n=s_i~\text{for infinitely many n}\}.\]
Clearly $Y_{s_i}\subset Y'$.
As cylinder sets are open we have $\mu(\I_{[s_i]})>0$, which by Birkhoff's Ergodic Theorem means that for $\mu$ a.e. $y\in Y$
\[\lim_{n\to\infty}\frac{1}{n}\sum_{j=0}^{n-1}\I_{[s_i]}\circ\sigma^j(y)>0,\]
but this is only possible if for $\mu$ a.e. $y$ we have $\I_{[s_i]}\circ\sigma^j(y)=1$ for infinitely many $j\in \N$, which means that $y_j=s_i$ for infinitely many $j\in \N$, which means $y\in Y_{s_i}$ for $\mu$ a.e. $y$. Thus $\mu(Y_{s_i})=1$, so $\mu(Y')=1$.
\end{proof}
As each $\mu_{\theta,\alpha}$ is a $g$-measure for $g\in \mathcal G_c$ it follows that there exists uncountably many distinct $K_\beta$ invariant probability measures on $(\Omega\times I_{\beta},\B_1\times B_2)$. In light of Theorem \ref{pressureiso} and Corollary \ref{RuelleCor} we can construct a $K_\beta$ invariant probability measure with arbitrary concentrated pressure on $\psi(Y')$. Furthermore each of these measures is strong mixing and has a Bernoulli natural extension.

\subsection{Novelty of the equilibrium states.}

Finally, we would like to verify that for any fixed $\beta\in B$ the measures $\psi_*\mu_{\theta,\alpha}$ are not just the $K_\beta$-invariant measures $m_{p}\times\mu_\beta$ that were already discovered in \cite{KarmaInv}. It was shown in \cite{KarmaInv} that for any choice of $p\in [0,1]$, the dynamical system $(\Omega\times I_\beta, \B_1\times \B_2,m_p\times\mu_\beta,K_\beta)$ is isomorphic to $(Y,\F,\mu_{P^{\beta,p}},\sigma)$, where $\mu_{P^{\beta,p}}$ is the Markov measure associated to the transition matrix $P^{\beta,p}$, defined by

\[P^{\beta,p}_{i,j}=\begin{cases}
	\lambda(C_i\cap T_\beta^{-1}C_j)/\lambda(C_i) & \text{if}~ i\in\cup_{k=0}^{b_1}M_k\\
	p & \text{if}~i\in\{0,\dots,K\}\backslash\cup_{k=0}^{b_1}M_k~\text{and}~j=0\\
	1-p & \text{if}~i\in\{0,\dots,K\}\backslash\cup_{k=0}^{b_1}M_k~\text{and}~j=K\\
	0 & \text{if}~~i\in\{0,\dots,K\}\backslash\cup_{k=0}^{b_1}M_k~\text{and}~j\notin\{0,K\}.
\end{cases}\]
and the stationary distribution $\pi^{\beta,p}$, which is strictly positive and satisfies $\pi^{\beta,p} P^{\beta,p}=\pi^{\beta,p}$. Proposition \ref{MarkG} thus gives that $\mu_{P^{\beta,p}}$ is a $g$-measure for the function $g_{P^{\beta,p}}(y)=\frac{\pi_{y_0}P_{y_0,y_1}}{\pi_{y_1}}$. As $(\Omega\times I_{\beta},\B_1\times B_2, \psi_*\mu_{\theta,\alpha},K_\beta)$ is measurably isomorphic to $(Y,\F,\mu_{\theta,\alpha},\sigma)$ for all choices of $\theta$ and $\alpha$, Lemma \ref{gmeassupport} (ii) implies that $\mu_{P^{\beta,p}}$ is equal to some $\mu_{\theta,\alpha}$ if and only if $g_{P^{\beta,p}}=g_{\theta,\alpha}$. Suppose that for some $\theta:\{0,\dots,\lbetar\}\to \R$ there exists a $m_p\times \mu_\beta$ such that $m_p\times \mu_\beta=\psi_*\mu_{\theta,\alpha}$. For convenience set $P=P^{\beta,p}$ and $\pi=\pi^{\beta,p}$. From the proof of Theorem \ref{uniquepot} this would then require, for all $y\in Y$,

\begin{equation}\label{uniqueequality}
	\frac{e^{\theta(e_0(y))}}{\sum_{i=0}^{\lfloor{\beta}\rfloor} e^{\theta(i)}}=\frac{\pi_{y_0}P_{y_0,y_1}}{\pi_{y_1}}.
\end{equation}

We now note that that Proposition \ref{partitionC} implies the following about the transition matrix $P$:

\begin{equation*}
	P_{i,0}=\begin{cases} \frac{1}{\beta}&\text{if}~i=0\\
	p&\text{if}~i\in\{0,\dots,K\}\backslash\bigcup_{i=0}^{\lbetar}M_{k}\\
	0&\text{otherwise}
	\end{cases},\quad
P_{i,K}=\begin{cases} \frac{1}{\beta}&\text{if}~i=K\\
	1-p&\text{if}~i\in\{0,\dots,K\}\backslash\bigcup_{i=0}^{\lbetar}M_{k}\\
	0&\text{otherwise}
\end{cases}.
\end{equation*} 
Since the elements of $\{0,\dots,K\}\backslash\bigcup_{i=0}^{\lbetar}M_{k}$ correspond to an entrance in the switch regions, we shall denote this set by $\{s_1,\cdots,s_{\lfloor \beta\rfloor}\}$.
As $\pi P=\pi$ this allows us to calculate the following equalities:

\begin{equation}\label{stationaryswitch}	
	\sum_{i=1}^{\lbetar}s_i=\frac{\beta-1}{p\beta}\pi_0=\frac{\beta-1}{(1-p)\beta}\pi_K
\end{equation}
The definition of $e_0$ along with our assumption \eqref{uniqueequality} means that for any $j\in\{1,\dots,\lbetar\}$ and $y\in Y$ with $y_0=s_j$ and $y_1=K$
\begin{equation}\label{eqnovelty1}
	\frac{e^{\theta(j-1)}}{\sum_{i=0}^{\lbetar}e^{\theta(i)}}=\frac{e^{\theta(e_0(s_j,K,y_2,\dots))}}{\sum_{i=0}^{\lbetar}e^{\theta(i)}}=\frac{(1-p)\pi_{s_j}}{\pi_K},
\end{equation}
and for any $y\in Y$ with $y_0=s_{\lbetar}$ and $y_1=0$,

\begin{equation}\label{eqnovelty2}
	\frac{e^{\theta(\lbetar)}}{\sum_{i=0}^{\lbetar}e^{\theta(i)}}=\frac{e^{\theta(e_0(s_{\lbetar},0,y_2,\dots))}}{\sum_{i=0}^{\lbetar}e^{\theta(i)}}=\frac{p\pi_{s_{\lbetar}}}{\pi_0}.
\end{equation}
Finally, we combine equations \eqref{stationaryswitch}, \eqref{eqnovelty1} and \eqref{eqnovelty2} to calculate:

\begin{align*}
	1&= \sum_{j=0}^{\lbetar}\frac{e^{\theta(j)}}{\sum_{i=0}^{\lbetar}e^{\theta(i)}}=\frac{(1-p)}{\pi_K}\sum_{j=1}^{\lbetar}\pi_{s_j}+\frac{p\pi_{s_{\lbetar}}}{\pi_0}\\
	& = \frac{\beta-1}{\beta}+\frac{p\pi_{s_{\lbetar}}}{\pi_0} = \frac{\beta-1}{\beta}+	\frac{e^{\theta(\lbetar)}}{\sum_{i=0}^{\lbetar}e^{\theta(i)}}.
\end{align*}
Thus if $g_{\theta,\alpha}=g_{P^{\beta,p}}$ for any $p\in (0,1)$, we must have that 
\[\frac{1}{\beta}=\frac{e^{\theta(\lbetar)}}{\sum_{i=0}^{\lbetar}e^{\theta(i)}},\]
but there are infinitely many $\theta$ for which this relationship does not hold, from which we can conclude that there are infinitely many measures $\psi_*\mu_{\theta,\alpha}$ which are not equal to any of the previously discovered $K_\beta$-invariant measures of the form $m_p\times \mu_\beta$.
\bibliography{DajaniPower}
\bibliographystyle{plain} 
 \end{document}